\documentclass[11pt]{amsart}
\usepackage{amsmath,amsthm, amscd, amssymb, color,amsfonts}
\usepackage[all]{xy}
\usepackage{mathrsfs}
\usepackage[inline]{enumitem}
\usepackage[T1]{fontenc}
\usepackage{hyperref}
\usepackage[symbol]{footmisc}

\usepackage{graphicx,fancyhdr}

\newcommand{\comment}[1]{}
\newcommand{\U}{\mathcal{U}}
\newcommand{\Gln}{\operatorname{GL}}

\numberwithin{equation}{section}
\newtheorem{theorem}{Theorem}[section]
\newtheorem{lemma}[theorem]{Lemma}

\newtheorem{prop}[theorem]{Proposition}

\newtheorem*{claim*}{Claim}
\newtheorem{paso}{Step}

\theoremstyle{definition}
\newtheorem{definition}[theorem]{Definition}

\theoremstyle{remark}

\newtheorem{remark}[theorem]{Remark}

\def\pf{\begin{proof}}
	\def\epf{\end{proof}}

\newcommand{\ku}{ \Bbbk}
\newcommand{\kut}{ \Bbbk^{\times}}

\newcommand{\ghost}{\mathscr{G}}
\newcommand{\ghostito}{\mathtt{j}}
\newcommand{\ghostvar}{M}

\newcommand\I{\mathbb I}

\newcommand\N{\mathbb N}
\newcommand{\bP}{\mathbb{P}}

\newcommand{\bq}{\mathbf{q}}

\newcommand\Z{\mathbb Z}

\newcommand{\zdos}{\Z^2}
\newcommand{\ld}[2]{\lambda_{#1}^{(#2)}}
\newcommand{\ldv}[2]{\mathtt L_{#1}^{(#2)}}
\newcommand{\ldsol}[2]{\ell_{#1}^{(#2)}}
\newcommand{\af}[2]{\zeta_{#1}^{(#2)}}

\renewcommand{\_}[1]{_{\left( #1 \right)}}

\newcommand{\quot}[2]{\mathscr{D}_{#2}{(#1)}}

\newcommand{\cA}{\mathcal{A}}

\newcommand\cK{\mathcal{K}}

\newcommand\cV{\mathcal{V}}

\newcommand{\rep}{\operatorname{rep}}

\newcommand\ngo{\mathfrak n}

\newcommand{\lstr}{\mathfrak L}

\newcommand\Aut{\operatorname{Aut}}

\newcommand{\id}{\operatorname{id}}

\newcommand{\Irr}{\operatorname{irrep}}
\newcommand{\IRR}{\operatorname{Irrep}}

\newcommand{\GK}{\operatorname{GKdim}}

\newcommand\ord{\operatorname{ord}}

\newcommand{\yd}[1]{{}^{ #1 }_{ #1 }\mathcal{YD}}

\newcommand{\toba}{\mathscr{B}}
\newcommand{\Vs}{\mathscr{V}}

\newcommand{\ot}{\otimes}

\allowdisplaybreaks

\newcounter{tabla}\stepcounter{tabla}

\begin{document}
	
	\title[On the Laistrygonian Nichols algebras]{On the Laistrygonian Nichols algebras \\ that are domains}
	
	\author[Andruskiewitsch, Bagio, Della Flora, Flores]
	{Nicol\'as Andruskiewitsch, Dirceu Bagio$^\ast$  , Saradia Della Flora, Daiana Flores}

	\address{FaMAF-Universidad Nacional de C\'ordoba, CIEM (CONICET), \newline
		Medina A\-llen\-de s/n, Ciudad Universitaria, (5000) C\' ordoba, \newline
		Rep\'ublica Argentina.} \email{nicolas.andruskiewitsch@unc.edu.ar}
	
	\address{Departamento de Matem\'atica, Universidade Federal de Santa Maria, \newline
		97105-900, Santa Maria, RS, Brazil} \email{dirceu.bagio@ufsm.br, saradia.flora@ufsm.br, flores@ufsm.br}

	\begin{abstract}
			We consider a class of Nichols algebras $\mathscr{B} (\mathfrak L_q( 1, \mathscr{G}))$ introduced in \cite{aah-triang}
			which are domains and have many favorable properties  like AS-regular and strongly noetherian.
			We classify their finite-dimensional simple modules and their point modules.
	\end{abstract}

	\thanks{
		$^\ast$ Corresponding author\newline
		N. A. was partially supported by CONICET: PIP 112201501000 71CO, Foncyt: PICT-2019-2019-03660 and Secyt (UNC)}

	\maketitle

	\setcounter{tocdepth}{2}
	
	\tableofcontents
	
	\section{Introduction}
	\subsection{The context}
	The classification of Nichols algebras with finite Gelfand-Kirillov dimension ($\GK$)
	over abelian groups, although not yet complete, has recently made significant progress; 
	see \cite{aah-triang,AAM0,angiono-garcia} and references therein. 
	In particular  those that are domains  are completely classified, see \cite[Theorem 1.4.1]{aah-triang} and \cite{AA}.  
	Beyond those coming from quantum groups with generic parameter and the Jordan plane, 
	the next examples  are the Laistrygonian Nichols algebras $\toba(\lstr_q( 1, \ghost))$, 
	where $\ghost \in \N$. The precise definition is recalled below but notice that there are other Laistrygonian Nichols algebras 
	that are not domains.
	The purpose of this paper is to study the Laistrygonian Nichols algebras $\toba(\lstr_q( 1, \ghost))$
	as algebras, rather than as braided Hopf algebras.
	
	For the importance of  Nichols algebras over abelian groups towards 
	the classification of  pointed Hopf algebras with finite $\GK$ see \cite{A-nilpotent}. 
	
	\subsection{Main results}
	In Section \ref{section:preliminaries}
	we recall the definition and basic properties of the algebras $\toba(\lstr_q( 1, \ghost))$ from \cite{aah-triang}. 
	Since they have a PBW basis, they are iterated Ore extensions and therefore 
	AS-regular and Cohen-Macaulay, see Proposition \ref{prop:Ore}. Our first main result is the classification of the finite-dimensional
	simple modules of $\toba(\lstr_q( 1, \ghost))$, see Theorem \ref{thm:irredu-lest}. 
	Here is the basic idea of the proof:  there is a surjective algebra map $\nu_{\ghost}$ from $\toba(\lstr_q( 1, \ghost))$ to the quantum plane
	$\ku_q[X,Y]$ (or the usual polynomial ring since $q=1$ is allowed). 
	We show that any finite-dimensional irreducible representation of $\toba(\lstr_q( 1, \ghost))$  
	factorizes through $\nu_{\ghost}$, and thus is known.
	This curious phenomenon appeared in other examples, see for instance \cite[Lemma 2.1]{Na} and \cite[Theorem 3.11]{adp}.
	In Section \ref{section:twisting} we discuss relations between different Laistrygonian Nichols algebras.
	Our second main result is the classification of the point modules of $\toba(\lstr_q( 1, \ghost))$, see Theorem \ref{thm:point-modules-les}. 
	
	\subsection*{Notations and conventions} \label{subsec:notation}
	We denote the natural numbers by $\N$ and $\N_0=\N\cup \{0\}$.
	If $k < t \in \N_0$, then we denote $\I_{k, t} = \{n\in \N_0: k\le n \le t \}$,
	and   $\I_{t} := \I_{1,t}$. We work over an algebraically closed field $\ku$ of characteristic 0.

	All modules are left. As usual, $\rep A$ is the category of finite-dimensional representations of an algebra $A$;
	the set of isomorphism classes of simple objects in $\rep A$ is denoted $\Irr A$.
	As usual we talk without distinctions of an element of $\Irr A$ or its representative.
	We use indistinctly the languages of representations and modules.
	The braided tensor category of left Yetter-Drinfeld modules over a Hopf algebra $H$ is denoted by $\yd{H}$.
	
	Our reference for Hopf algebras is \cite{Rad-libro}. We use the expression \emph{braided Hopf algebra} as in \cite{Tk}; that is a (rigid) braided vector space with compatible multiplication and comultiplication. 
	As explained in \emph{loc. cit.} this means that it can be realized as a Hopf algebra in the braided tensor category $\yd{H}$
	over some Hopf algebra $H$.

	\section{Preliminaries}\label{section:preliminaries}
	\subsection{The Nichols algebra \texorpdfstring{$\toba(\lstr_q( 1, \ghost))$}{}}\label{subsection:lstr-11disc}
	
	We introduce the algebra of our interest; see \cite[\S 4.3.1]{aah-triang} for details.
	Let $\ghost \in \N$ and $q \in \ku^{\times}$. The algebra
	$\toba(\lstr_q( 1, \ghost))$ is presented by generators $x_1,x_2, (z_n)_{ n \in \I_{0,\ghost}}$ with defining relations
	\begin{align}\label{eq:rels B(V(1,2))}
		x_2x_1 &- x_1x_2+\frac{1}{2}x_1^2, \\   
		x_1 z_0 &- q \, z_0 x_1,  \label{eq:lstr-rels&11disc-1}  \\
		z_nz_{n+1}&- q^{-1} \, z_{n+1}z_n, &   n \in \I_{0,\ghost-1}, \label{eq:lstr-rels&11disc-2} 
		\\
		\label{eq:-1block+point-z2}
		x_2z_n &- qz_nx_2 - z_{n+1}, &  n \in \I_{0,\ghost-1},
		\\
		x_2z_\ghost &- qz_\ghost x_2.  \label{eq:lstr-rels&11disc-qserre}
	\end{align}
	$\toba(\lstr_q( 1, \ghost))$ is a domain and has a PBW-basis
	\begin{align*}
		B_{\ghost} = \{ x_1^{m_1} x_2^{m_2} z_{\ghost}^{n_{\ghost}} \dots z_1^{n_1} z_0^{n_0}: m_i, n_j \in\N_0\};
	\end{align*}
	hence $\GK \toba(\lstr_q( 1, \ghost)) = 3+\ghost$.
	The algebra $\toba(\lstr_q( 1, \ghost))$ is graded, with
	\begin{align}
		\label{eq:grading}
		\deg x_1 &= \deg x_2 = 1,& \deg z_n &= n+1, &  n \in \I_{0,\ghost}.
	\end{align}
	
	Actually, $\toba(\lstr_q( 1, \ghost))$ is the Nichols algebra of
	the braided vector space $\lstr_q( 1, \ghost)$ that has a basis $x_1, x_2, x_3 := z_0$, cf. \cite[\S 4.1.1]{aah-triang}
	and Section \ref{section:twisting} below.
	Indeed \eqref{eq:-1block+point-z2} is just the recursive definition of $z_n$ in terms of $x_1, x_2,z_0$.
	Notice that $q = q_{12} = q_{21}^{-1}$ 
	in the notation of  \cite{aah-triang}. In \emph{loc. cit.} the parameter $q$ was somehow neglected as the main focus
	was on the classification problem, see also Proposition \ref{prop:twisting}.  But for the sake of the algebraic properties
	the role of $q$ is central, as we see in this paper.

	\begin{remark} \label{ob-sub-jp}
		Notice that the subalgebra of $\toba(\lstr_q( 1, \ghost))$ generated by $x_1$ and  $x_2$ is isomorphic to the Jordan plane
		and has defining relation \eqref{eq:rels B(V(1,2))}.
	\end{remark}

	It follows from \eqref{eq:rels B(V(1,2))} by a standard argument that
	\begin{align}\label{eq:jordan}
		x_1x_2^j&= \big(x_2+\frac{1}{2}x_1 \big)^jx_1, & j &\in \N.
	\end{align}
	Also, one derives from \cite[Lemmas 4.3.3, 4.3.4]{aah-triang} that
	\begin{align}
		\label{eq:-1block+point}
		x_1z_n &=  qz_nx_1, &  n \in \I_{0,\ghost},
		\\
		\label{eq:bracket ztzk}
		z_mz_n&=  q^{m-n}  z_nz_m, &  m<n\in  \I_{0,\ghost}.
	\end{align}
	
	\subsection{Ring-theoretical properties} We  show that $\toba(\lstr_q( 1, \ghost))$ is an iterated Ore extension.
	Hence  it is strongly noetherian
	by \cite[Proposition  4.10]{ASZ};  AS-regular by \cite[Proposition 2]{AST} and Cohen-Macaulay by \cite[Lemma 5.3]{ZZ}.
	
	\medbreak
	We start by an auxiliary result. Consider the following subalgebras of $\toba(\lstr_q( 1, \ghost))$:
	$R_1 = \ku[x_1]$, $R_2 = \ku\langle x_1, x_2\rangle$, $R_3 = \ku\langle x_1, x_2, z_{\ghost}\rangle$ and in general
	\begin{align*}
		R_{\ghost +3 - \ghostito} &= \ku\langle x_1, x_2, z_{\ghost}, z_{\ghost - 1}, \dots, z_{\ghostito}\rangle, & \ghostito&\in \I_{0, \ghost}.
	\end{align*}  
	Let $\ghostito\in \I_{0, \ghost}$.  Because of the defining relations,
	\eqref{eq:-1block+point} and \eqref{eq:bracket ztzk} we have that
	\begin{align*}
		B_{\ghost +3 - \ghostito} =\{ x_1^{m_1} x_2^{m_2} z_{\ghost}^{n_{\ghost}} \dots z_\ghostito^{n_\ghostito}: m_i, n_h \in\N_0\}
	\end{align*}
	is a PBW-basis of $R_{\ghost +3 - \ghostito}$.  
	
	Let now $\ghostito \in \I_{1, \ghost -1}$. We denote by $\underline{x}_1,\underline{x}_2, 
	(\underline{z}_n)_{ n \in \I_{0,\ghost -\ghostito}}$ 
	the generators of $\toba(\lstr_q( 1, \ghost - \ghostito))$  and by $\underline{B}_{\ghost - \ghostito}$
	the corresponding PBW-basis. 
	Then there is an algebra map $\psi: \toba(\lstr_q( 1, \ghost - \ghostito)) \to R_{\ghost +3 - \ghostito}$ given by
	\begin{align*}
		\underline{x}_1 & \mapsto x_1, &\underline{x}_2& \mapsto x_2, &
		\underline{z}_n & \mapsto z_{\ghostito + n}, & n &\in \I_{0, \ghost -\ghostito}.
	\end{align*}
	Indeed it is easy to see that this assignement preserves the defining relations.

	\begin{lemma}\label{lema:subalgebra-lestrygon}
		The map  $\psi: \toba(\lstr_q( 1, \ghost - \ghostito)) \to R_{\ghost +3 - \ghostito}$ is an  isomorphism.
	\end{lemma}
	\pf 
	Clearly $\psi$  sends the PBW-basis $\underline{B}_{\ghost - \ghostito}$  to the PBW-basis $B_{\ghost +3 - \ghostito}$ .
	\epf
	
	To fix the notation, we recall that given  a ring $R$,  $ \sigma \in \Aut (R)$ and 
	a $(\sigma, 1)$-derivation $\delta$ of $ R$, i.~e.
	$ \delta (rr')=\sigma (r)\delta (r')+\delta (r)r'$,
	the Ore extension 
	$ R[X;\sigma ,\delta ]$ (or simply $R[X; \sigma]$  if $\delta = 0$)
	is the ring of polynomials $R[X]$ with  the multiplication determined by
	$ Xr=\sigma (r)X+\delta (r)$, $r\in R$.

	\begin{prop}\label{prop:Ore} 
		The algebra $\toba(\lstr_q( 1, \ghost))$ is an iterated Ore extension.
	\end{prop}

	\pf  
	It is well-known that  $R_2$ is an  Ore extension of $R_1$ and it follows easily that
	$R_3 \simeq R_2[X; \sigma_\ghost]$ where $\sigma_\ghost(x_1) = q^{-1} x_1$,  $\sigma_\ghost(x_2) = q^{-1} x_2$.
	Let  $\ghostito \in \I_{0, \ghost -1}$. 
	By the preceding discussion,  $R_{\ghost +3 - \ghostito}$ is a free $R_{\ghost + 2 - \ghostito}$-module with basis $(z_\ghostito^{n})_{n\in\N_0}$.
	Using Lemma \ref{lema:subalgebra-lestrygon} we check that there are  an  algebra automorphism 
	$\sigma_{\ghostito}$  and a $(\sigma_{\ghostito}, 1)$-derivation $\delta_{\ghostito}$ of $R_{\ghost +2 - \ghostito}$
	determined by  
	\begin{align*}
		\sigma_{\ghostito} (x_1) &= q^{-1} x_1, & \sigma_{\ghostito} (x_2) &= q^{-1} x_2,& 
		\sigma_{\ghostito} (z_i) &= q^{\ghostito - i} z_i, 
		\\
		\delta_{\ghostito} (x_1) &= 0, & \delta_{\ghostito} (x_2) &=  -q^{-1} z_{\ghostito + 1},& 
		\delta_{\ghostito} (z_i) &= 0,
	\end{align*}
	$i\in \I_{\ghostito +1, \ghost}$.
	Therefore 
	$R_{\ghost +3 - \ghostito}\simeq R_{\ghost +2 - \ghostito}[X; \sigma_{\ghostito}, \delta_{\ghostito}]$,  
	for all $j\in \I_{0,\ghost-1}$. \epf

	\subsection{Quotients of the Laistrygon}  The notion of exact sequence of Hopf algebras in braided tensor categories was first discussed in \cite{BeD}. 
	In the particular setting of braided Hopf algebras as in \cite{Tk}, the first reference we are aware of  is  \cite{andrus-natale}.
	We recall from \emph{loc. cit.} that the sequence of braided Hopf algebras and
	braided Hopf algebra morphisms
	\begin{align*}
		\xymatrix@C=10mm@R=9mm{   
			0\ar[r]&S \ar[r]^{\iota}& R\ar[r]^{\pi} &T \ar[r]&0}
	\end{align*}
	is exact if $\iota$ is injective, $\pi$ is surjective, $\ker\pi\overset{\star}{=}RS^{+}$ and $R^{{\rm co}\,\pi}=S$. A braided Hopf algebra $R$ fitting into the previous exact sequence is called an extension of $T$ by $S$.  Clearly $\star$ implies
	the equality $RS^{+} = S^{+}R$; when this last equality holds, we say that $S$ is normal in $R$.
	Notice that there are exact sequences where either $S$ or $T$, or both, are usual Hopf algebras but $R$ is braided in a strict sense.
	
	We now present $\toba(\lstr_q( 1, \ghost))$ as an extension of braided Hopf algebras.
	By \eqref{eq:jordan} and \eqref{eq:-1block+point}, we have that
	\begin{align*}
		\toba(\lstr_q( 1, \ghost)) x_1\toba(\lstr_q( 1, \ghost))=\toba(\lstr_q( 1, \ghost)) x_1.
	\end{align*}
	Hence $\ku [x_1]$ is a normal Hopf subalgebra of $\toba(\lstr_q( 1, \ghost))$ and
	\begin{align*}
		\quot{\ghost}{q} :=\toba(\lstr_q( 1, \ghost))/\toba(\lstr_q( 1, \ghost)) x_1\toba(\lstr_q( 1, \ghost))
	\end{align*}
	is a braided Hopf algebra quotient
	that fits into an exact sequence
	\begin{align*}
		0\to \ku [x_1]\to \toba(\lstr_q( 1, \ghost)) \overset{\varpi}{\longrightarrow}  \quot{\ghost}{q} \to 0
	\end{align*} 
	of braided Hopf algebras.
	Using the PBW-basis we see that $ \quot{\ghost}{q}$ is generated by  $x_2, (z_n)_{ n \in \I_{0,\ghost}}$ with defining relations  \eqref{eq:lstr-rels&11disc-2}, \eqref{eq:-1block+point-z2} and \eqref{eq:lstr-rels&11disc-qserre}. 
	Here and below we use the same notation for $x_2, z_n$ and their images in $ \quot{\ghost}{q}$.
	
	\medbreak
	The projection $\varpi$ above induces a map 
	$\Irr  \quot{\ghost}{q} \to \Irr \toba(\lstr_q( 1, \ghost))$. 
	
	\begin{lemma}\label{rem:irr-L(1,g)} The above map is bijective: 
		$\Irr \toba(\lstr_q( 1, \ghost)) \simeq \Irr  \quot{\ghost}{q}$.
	\end{lemma}
	
	\pf
	
	By Remark \ref{ob-sub-jp}, $x_1$ and $x_2$ generate a subalgebra of $\toba(\lstr_q( 1, \ghost))$ isomorphic to the Jordan plane. Thus, by \cite[Lemma 2.1]{Na} we have that $x_1$ acts nilpotently on every finite-dimensional $\toba(\lstr_q( 1, \ghost))$-module; 
	but $\ker x_1$ is a submodule by the preceding, hence
	$x_1$ acts by 0 on every finite-dimensional simple $\toba(\lstr_q( 1, \ghost))$-module.
	\epf

	\begin{lemma}\label{lem:epi-algebra} Let $\ghost >1$. The map $\pi_{\ghost}:\toba(\lstr_q( 1, \ghost))\to \toba(\lstr_q( 1, \ghost-1))$
		given by 
		\begin{align*}
			&\pi_{\ghost}(x_i)=x_i,& &\pi_{\ghost}(z_j)=z_j,& &\pi_{\ghost}(z_{\ghost})=0,& &i\in \I_2,& &j\in \I_{0,\ghost-1},&
		\end{align*}
		is an algebra epimorphism.\qed
	\end{lemma}

	Clearly $\ker \pi_{\ghost}=z_{\ghost}\toba(\lstr_q( 1, \ghost))$, thus we have an isomorphism of algebras
	\begin{align}
		\label{eq:zghost-quotient}
		\toba(\lstr_q( 1, \ghost))/z_{\ghost}\toba(\lstr_q( 1, \ghost))\simeq \toba(\lstr_q( 1, \ghost-1)).
	\end{align}

	
	\section{Simple Modules of \texorpdfstring{$\toba(\lstr_q( 1, \ghost))$}{} }\label{section:simple-modules}
	The purpose of this section is to give the classification of the finite-dimensional simple $\toba(\lstr_q( 1, \ghost))$-modules. We reduce this computation to those of the quantum plane, see Proposition \ref{prop:simple_quantum-plane}.

	\subsection{Simple modules of the quantum plane}\label{subsec:quantum-plane}
	Let $\ku_q[X,Y]$ denote the   algebra generated by $X$ and $Y$ with defining relation $XY- qYX$. 
	Then there is  a surjective algebra map $\nu_{\ghost}:\toba(\lstr_q( 1, \ghost))\to \ku_q[X,Y]$
	given by 
	\begin{align*}
		\nu_{\ghost}(x_1) &= \nu_{\ghost}(z_j)=0,&  j &\in \I_{\ghost},&\nu_{\ghost}(x_2) &=X,& \nu_{\ghost}(z_0) &= Y
	\end{align*}
	for any $\ghost \in \N$. Clearly  $\nu_{\ghost} = \nu_1\pi_2\ldots\pi_{\ghost}$, cf. Lemma \ref{lem:epi-algebra}.
	
	\medbreak
	If $q =1$, then $\ku_1[X,Y] = \ku[X,Y]$   is the polynomial ring in 2 variables; by Hilbert's Nullstellensatz
	its finite-dimensional simple modules are all one-dimensional and parametrized by the points of the plane.
	Assume that $q \neq 1$; then $\ku_q[X,Y]$ is called the quantum plane of parameter $q$.
	We recall the well-known classification of its finite-dimensional simple modules.
	First, there are the one-dimensional  $\ku_q[X,Y]$-modules $\ku^X_{a} = \ku$  with action
	$X \cdot 1 = a$, $Y \cdot 1 = 0$ and $\ku^Y_{a} = \ku$  with action  $X \cdot 1 = 0$, $Y \cdot 1=a$, for every 
	$a\in \ku^{\times}$.
	Second, suppose that $\ord q =:N < \infty$ and let $(e_i)_{i \in\I_N}$ be  the canonical basis of $\ku^N$. 
	Given $a,b\in \ku^{\times}$,  the $\ku_q[X,Y]$-module $\U_{a,b}$ is $\ku^N$ with the action defined by
	\begin{align*}
		X e_i  = aq^{i-1}e_i,& & Ye_j  =e_{j+1},& & Y e_N = be_1,& & i \in \I_{N},\, j\in \I_{N-1}.
	\end{align*}
	It is easy to see that $\U_{a,b}$ is simple.
	
	\begin{prop}\label{prop:simple_quantum-plane} Assume that $q \neq 1$.
		Let $V\in \Irr\ku_q[X,Y]$. 
		\begin{enumerate}[leftmargin=*,label=\rm{(\alph*)}]
			\item If $\dim V = 1$, then $V$ is isomorphic to $\ku^X_{a}$, or to  $\ku^Y_{a}$ for a unique $a\in \ku^{\times}$.
			
			\item If $\dim V > 1$, then  $\ord q =:N < \infty$ and $V \simeq \U_{a,b}$, 
			for unique $a,b\in \ku^{\times}$.
		\end{enumerate}
	\end{prop}
	\pf Since $\ker X$ is a  $\ku_q[X,Y]$-submodule of $V$, then $\ker X=V$ or $0$. If $\ker X=V$, then $V=\langle v\rangle$ is one-dimensional, $Yv=av$ and $V\simeq \ku^Y_{a}$ for a unique $a\in \ku^{\times}$. If $\ker X=0$, then from $XY=qYX$ we see that 
	$(1-q^{\dim V}) \det Y =0$. If $\det Y=0$, then $V\simeq \ku^X_{a}$ for a unique $a\in \ku^{\times}$ as before. If $\det Y\neq 0$, then $q^{\dim V}=1$, so that $\ord q < \infty$. Since $XY^N = Y^{N}X$, there exist $v\in V - 0$ and $a,b\in \ku^{\times}$ such that $Xv=av$ and $Y^N v=bv$. Therefore $V= \langle Y^iv: i \in \I_{0,N-1}\rangle$ and consequently $V\simeq \U_{a,b}$. \epf
	
	\begin{remark}
		The infinite-dimensional simple $\ku_q[X,Y]$-modules were computed in \cite{Ba} using results from \cite{BO}.
	\end{remark}

	\subsection{Finite-dimensional simple modules}
	We proceed now with the classification of the finite-dimensional simple $\toba(\lstr_q( 1, \ghost))$-modules. 
	
	Recall that $ \quot{\ghost}{q}$ is generated by  $x_2, (z_n)_{ n \in \I_{0,\ghost}}$ with defining relations  \eqref{eq:lstr-rels&11disc-2}, \eqref{eq:-1block+point-z2} and \eqref{eq:lstr-rels&11disc-qserre}. The relations \eqref{eq:-1block+point-z2} and \eqref{eq:lstr-rels&11disc-qserre} implies that
	\begin{align}\label{eq1}
		z_{\ghost-1}x_2^j=q^{-j}x_2^jz_{\ghost -1}-jq^{-j}x_2^{j-1} z_{\ghost}, && j \in \N.
	\end{align}
	
	We start with an auxiliary result.
	\begin{lemma} \label{aux1}
		Let $V \in \rep  \quot{\ghost}{q}$, $n = \dim V$. If the action of $z_{\ghost}$ is invertible, 
		then the actions of $z_{\ghost-1}$, $x_2$ are invertible and $q^n = 1$.
	\end{lemma}
	\pf We prove that $z_{\ghost-1}$ is invertible; the proof for $x_2$ is similar. Suppose that $\ker z_{\ghost-1}\neq 0$. 
	Note that $z_{\ghost} \ker  z_{\ghost-1}\subseteq \ker z_{\ghost-1}$ by \eqref{eq:lstr-rels&11disc-2}. Hence, there exist $\lambda\in \ku^{\times}$ and
	$0 \neq v_0\in \ker z_{\ghost-1}$ such that $z_{\ghost} v_0=\lambda v_0$. Let $v_j:=x_2^j v_0$,  $j\in \N_0$.
	By \eqref{eq:lstr-rels&11disc-qserre},  $z_{\ghost} v_j = \lambda q^{-j} v_j$.
	By \eqref{eq1},
	\begin{align}\label{actionz0-new1}
		z_{\ghost-1} v_j &= -j \lambda q^{-j} v_{j-1}, & j&\in \N.
	\end{align}
	(This is also valid for $j =0$ if we agree that $v_{-1} =0$).
	Since $\dim V < \infty$, there exists $m\in \N$
	such that $v_m \in \langle v_j:  j \in \I_{0, m-1} \rangle$. Pick $m$ minimal (here we use that $v_0 \neq 0$) and write 
	$v_m = \sum_{j \in \I_{0, m-1}} a_j v_j$. Applying $-z_{\ghost-1}$, we see from \eqref{actionz0-new1}
	that
	\begin{align*}
		m \lambda q^{-m}  v_{m-1} &= \sum_{j \in \I_{0, m-1}} j \lambda q^{-j}a_j  v_{j-1}.
	\end{align*}
	Since $\lambda \neq 0$, we conclude that $v_{m-1} \in \langle v_j: j \in \I_{0, m-2} \rangle$,
	a contradiction to the minimality of $m$. Hence $z_{\ghost-1}$ is invertible.
	From \eqref{eq:lstr-rels&11disc-2} follows that 
	$(1-q^{n})\det z_{\ghost}\det z_{\ghost-1}=0$ and consequently $q^n =1$.
	\epf

	\begin{lemma}\label{lema:aux2}
		Let $V\in \Irr  \quot{\ghost}{q}$. Then $z_{\ghost}=0$ on $V$.
	\end{lemma}
	
	\pf Let $V \in \Irr  \quot{\ghost}{q}$, $n=\dim V$. Then $\ker z_{\ghost}$ is a submodule of
	$V$ by \eqref{eq:lstr-rels&11disc-qserre} and \eqref{eq:bracket ztzk};
	consequently $\ker z_{\ghost}=0$ or $\ker z_{\ghost}=V$. Suppose that  $\ker z_{\ghost}=0$. 
	Then $q^n=1$ and $\ker z_{\ghost -1} = \ker x_2=0$ by Lemma \ref{aux1}. Hence $n=lN$, where $N=\ord q$ and $l \in \N$. Let $\lambda\in \ku^{\times}$ an eigenvalue of $z_{\ghost}$ and $\lambda_j:=\lambda q^j$, $j\in \Z$. 
	Let $V^{\lambda} = \ker (z_{\ghost }- \lambda)$ denote the eigenspace of eigenvalue $\lambda$.
	By \eqref{eq:lstr-rels&11disc-2}, $z_{\ghost -1}V^{\lambda_j} \subseteq V^{\lambda _{j+1}}$, $j \in \I_{0, N-1}$. Since $z_{\ghost -1}$ is invertible, $z_{\ghost -1}V^{\lambda_j} = V^{\lambda_{j+1}}$. Similarly, $x_2V^{\lambda_j} = V^{\lambda_{j-1}}$. Thus, $V= V^{\lambda_{0}}\oplus\ldots\oplus  V^{\lambda_{N-1}}$ and $\dim  V^{\lambda_{j}}=l$. Let $\mathtt{B}_{\lambda}=\{v_j\,:\,j\in \I_l\}$ be a basis of $V^{\lambda}$. Then $\mathtt{B}=\cup_{i\in \I_{0,N-1}}z_{\ghost -1}^{i}\mathtt{B}_{\lambda}$ is a basis of $V$ and the actions of $z_{\ghost -1}$, $z_{\ghost}$ and $x_2$ in this basis are given, respectively, by the following matrices
	\begin{align*}
		&\begin{pmatrix}
			0 & \ldots & 0 & A\\
			\id_l & \ldots & 0 & 0 \\
			\vdots & \ddots & \vdots & \vdots\\
			0 & \ldots & \id_l & 0
		\end{pmatrix}, &  
		&\begin{pmatrix}
			\lambda\id_l  & \ldots & 0\\
			0 & \ldots & 0\\
			\vdots & \ddots & \vdots\\
			0 &  \ldots & \lambda_ {N-1}\id_l
		\end{pmatrix}, &
		&\begin{pmatrix}
			0 & B_2 & \ldots & 0\\
			\vdots & \vdots & \ddots & \vdots\\
			0 & 0 & \ldots & B_N\\
			B_1 & 0 & \ldots & 0
		\end{pmatrix},
	\end{align*}
	with $A, B_i \in \Gln_l(\ku)$, $i \in \I_{N}$. By \eqref{eq:-1block+point-z2}, we have for all $ j \in \I_{2,N-1}$,
	\begin{align*}
		B_2 &= qAB_1+\lambda \id_l,&  
		B_{j+1} &= qB_j+\lambda_{j-1}\id_l, &
		B_1A &=qB_N+\lambda_{N-1}\id_l.
	\end{align*}
	Arguing inductively we see that $AB_1+N\lambda_{N-1}\id_l=B_1A$. Applying the trace map to this  identity we get that $n\lambda_{N-1}=0$, thus $\lambda=0$, a contradiction.
	\epf
	
	\begin{theorem}\label{thm:irredu-lest} $\Irr \toba(\lstr_q( 1, \ghost)) \simeq \Irr  \ku_q[X,Y].$
	\end{theorem}
	\pf Let  $V \in \Irr \toba(\lstr_q( 1, \ghost))$. By Lemma \ref{rem:irr-L(1,g)}, $V \in \Irr  \quot{\ghost}{q}$ and thus by Lemma \ref{lema:aux2}, $z_{\ghost}=0$ on $V$.  Notice that $ \quot{\ghost}{q} z_{\ghost} \quot{\ghost}{q}=z_{\ghost} \quot{\ghost}{q}$. Since $ \quot{\ghost-1} {q}
	=   \quot{\ghost}{q}/ \quot{\ghost}{q} z_{\ghost} \quot{\ghost}{q}$, using Lemma \ref{lema:aux2} again, we conclude that $z_{\ghost-1}=0$ on $V$. Repeating this $\ghost$-times, we see that $V \in \Irr \ku_q[X,Y]$. \epf

	\section{Twisting and isomorphisms}\label{section:twisting}

	\subsection{Twisting}
	
	In this Subsection, following \cite{AS1}, we use the term twisting to refer to the twisting of the multiplication
	introduced in \cite{DT} which is dual to the twisting of the comultiplication in an appropriate sense.
	Precisely, let $H$ be a Hopf algebra  and  $\sigma : H \otimes H \to \ku$ be an invertible 2-cocycle. Consider the Hopf algebra $H_{\sigma}$ which has the same coalgebra structure of $H$ and multiplication  given by
	\begin{align}\label{eq:twist-mult}
		x \cdot_{\sigma }y &= \sigma(x\_{1}, y\_{1})x\_{2}y\_{2}\sigma^{-1}(x\_{3}, y\_{3}), &x,y &\in H;
	\end{align}
	$H_{\sigma}$ is obtained by twisting the multiplication of $H$.
	
	\smallbreak
	We start by a definition implicit in \cite[\S 2.4]{AS1}.
	Let $R$ be a Hopf algebra in $\yd{H}$, $A := R \#H$, $\pi : A \to H$ and $\iota: H \to A$ 
	be the canonical projection and injection. Define $\sigma^\pi : A \otimes  A \to \ku$ by $\sigma^\pi :=\sigma(\pi \otimes \pi)$. Since
	the maps $\pi : A_{\sigma^{\pi}} \to H_{\sigma}$ and $\iota: H_{\sigma} \to A_{\sigma^{\pi}}$ are still Hopf algebra maps
	and the comultiplication is not changed, $A_{\sigma^{\pi}} \simeq R_{\sigma} \# H_{\sigma}$
	where $R_{\sigma}$  is a Hopf algebra in $\yd{H_{\sigma}}$ that coincides with $R$ as vector subspace of $A$, with multiplication
	\begin{align}\label{eq:twist-mult-R}
		x\cdot_{\sigma } y &= \sigma(x\_{0}, y\_{0})x\_{1}y\_{1}, & x,y &\in R_{\sigma}.
	\end{align}

	\begin{definition} Let $R$ and $S$  be braided Hopf algebras in the sense of \cite{Tk}.
		We say that $R$ and $S$ are \emph{twist-equivalent}
		if there exist a Hopf algebra $H$ and an invertible 2-cocycle $\sigma : H \otimes H \to \ku$ such that
		\begin{itemize}[leftmargin=*]
			\item $R$ is realizable in $\yd{H}$;

			\item $S$ is isomorphic to $R_{\sigma}$ as a braided Hopf algebra.
		\end{itemize}
	\end{definition}
	
	The notion of twist-equivalence is useful for classification purposes.
	
	\begin{lemma} \label{lema:twisting-nichols}
		\cite[Lemma 2.13]{AS1} Let $H$ and $\sigma$ be as above.
		If $R = \oplus_{n\in \N_0}R(n)$ is a  graded Hopf algebra in $\yd{H}$, then
		$R_{\sigma}$ is a  graded Hopf algebra in $\yd{H_{\sigma}}$
		with $R(n)=R_{\sigma}(n)$ as vector spaces, for all $n \geq 0$. 
		Moreover,  $R$ is a Nichols algebra if and only if $R_\sigma$ is a Nichols algebra. \qed
	\end{lemma}

	We recall  that two matrices  $\bq = (q_{ij})_{i,j \in \I_{\theta}}$ and 
	$\bq' = (q'_{ij})_{i,j \in \I_{\theta}}$ with entries in $\kut$
	are \emph{twist-equivalent} if 
	\begin{align*}
		q_{ii} &= q'_{ii}& &\text{and}&  q_{ij}q_{ji} &= q'_{ij}q'_{ji},& &\text{for all } i \neq j \in \I_{\theta}.
	\end{align*}
	See \cite[Definition 3.8]{AS1}.
	Suppose that this is the case.
	Let $V$ and $V'$ be braided vector spaces of diagonal type with braiding matrices $\bq$ and $\bq'$ respectively.
	Then \cite[Proposition 3.9]{AS1} essentially says that the Nichols algebras $\toba(V)$ and $\toba(V')$
	are twist-equivalent. Our first goal in this Subsection is to extend this result to a class of braided vector spaces of dimension 3.
	
	\medbreak
	More precisely, let $(q_{ij})_{i,j \in \I_2}$ be a matrix of non-zero scalars and $a\in \ku$. Let
	$\Vs$ be the braided vector space with basis $(x_i)_{i\in\I_3}$ and braiding
	\begin{align}\label{eq:braiding-block-point}
		(c(x_i \otimes x_j))_{i,j\in \I_3} &=
		\begin{pmatrix}
			q_{11} x_1 \otimes x_1&  q_{11}(x_2 + x_1) \otimes x_1&  q_{12} x_3  \otimes x_1
			\\
			q_{11} x_1 \otimes x_2 & q_{11} (x_2 + x_1) \otimes x_2& q_{12} x_3  \otimes x_2
			\\
			q_{21} x_1 \otimes x_3 &  q_{21} (x_2 +a x_1) \otimes x_3&  q_{22} x_3  \otimes x_3
		\end{pmatrix}.
	\end{align} 
	We realize $\Vs$ in $\yd{\Z^{2}}$ as follows. If $\alpha_{1}, \alpha_{2}$ is the canonical basis
	of $\zdos$, then the action of $\zdos$ on $\Vs$ and the $\zdos$-grading are given by
	\begin{align}\label{eq:YD-structure}
		\begin{aligned}
			\alpha_{1} \rightharpoonup x_1 &= q_{11} x_{1}, & \alpha_{1} \rightharpoonup x_2 &= q_{11} (x_2 + x_1), & \alpha_{1} \rightharpoonup x_3 &= q_{12} x_{3};
			\\
			\alpha_{2} \rightharpoonup x_1 &= q_{21} x_{1}, & \alpha_{2} \rightharpoonup x_2 &= q_{21} (x_2 + a x_1), & \alpha_{2} \rightharpoonup x_3 &=  q_{22} x_{3};
			\\
			\deg x_1 &= \alpha_{1}, & \deg x_2 &= \alpha_{1}, & \deg x_3 &= \alpha_{2}.
		\end{aligned}
	\end{align}
	Then the Nichols algebra $\toba(\Vs)$ is a Hopf algebra in $\yd{\Z^{2}}$
	and we may consider the bosonization $\cA = \toba(\Vs) \# \ku \zdos$, used in the proof below.
	
	\medbreak
	Let $\Vs'$ be the braided vector space with basis $(x_i)_{i\in\I_3}$ and braiding \eqref{eq:braiding-block-point}
	but with respect to $(q'_{ij})_{i,j \in \I_2}$  and the same $a\in \ku$. 
	Assume that $(q_{ij})$ and $(q'_{ij})$ are twist-equivalent, i.~e. 
	$q_{11} = q'_{11}$, $q_{22} = q'_{22}$ and $q_{12}q_{21} = q'_{12}q'_{21}$.

	\begin{lemma} \label{lem:twisting}
		The Nichols algebras $\toba(\Vs)$ and $\toba(\Vs')$ are twist-equivalent.
	\end{lemma}
	
	\pf  We argue as in \cite[Lemma 2.12]{AS1}.
	Let $(p_{ij})_{i,j \in \I_2} \in (\ku^{\times})^{2 \times 2}$.
	Let $\sigma: \zdos \times \zdos \to \ku^{\times}$ be the  bilinear form, hence a 2-cocycle, given by
	$\sigma(\alpha_{i}, \alpha_{j}) =   p_{ij}$, that
	we extend to an invertible 2-cocycle $\sigma: \ku \zdos \otimes \ku \zdos \to \ku$
	with the same name.
	Let us twist the multiplication of  $\cA$ by the cocycle $\sigma^{\pi} := \sigma(\pi \ot \pi)$
	where $\pi: \cA \to \ku\zdos$ is the natural projection. 
	Then $\cA_{\sigma^{\pi}} = \toba(\Vs)_{\sigma} \# \ku \zdos$ where 
	$\toba(\Vs)_{\sigma} \in \yd{ \Z^{2}}$ has  the same $\N_0$-grading as $\toba(\Vs)$. 
	As object of  $\yd{\cA_{\sigma}}$,  the  coaction of $\ku \zdos $ on $\toba(\Vs)_{\sigma}$ (i.e. the $\zdos$-grading)
	coincides with the coaction on $\toba(\Vs)$,
	while the action of $\ku\zdos$ on $\toba(\Vs)_{\sigma}$ is determined by
	\begin{align}\label{eq:YD-structure-twisted}
		\begin{aligned}
			\alpha_{1} \rightharpoonup_{\sigma} x_1 &= q_{11} x_{1}, & 
			\alpha_{2} \rightharpoonup_{\sigma} x_1 &= p_{21}  p_{12}^{-1} q_{21} x_{1},
			\\
			\alpha_{1} \rightharpoonup_{\sigma} x_2 &= q_{11} (x_2 + x_1), & 
			\alpha_{2} \rightharpoonup_{\sigma} x_2 &= p_{21}  p_{12}^{-1} q_{21} (x_2 + a x_1), 
			\\\alpha_{1} \rightharpoonup_{\sigma} x_3 &= p_{12}  p_{21}^{-1}  q_{12} x_{3},
			& \alpha_{2} \rightharpoonup_{\sigma} x_3 &=  q_{22} x_{3}.
		\end{aligned}
	\end{align}
	Indeed, observe that 
	\begin{align*}
		\Delta^{2}(x_i) & = x_i \otimes 1 \otimes 1 +  \alpha_1 \otimes x_i \otimes 1 + \alpha_1  \otimes\alpha_1 \otimes x_i ,
		& i \in \I_2;\\
		\Delta^{2}(x_3) & = x_3 \otimes 1 \otimes 1 +  \alpha_2 \otimes x_3 \otimes 1 + \alpha_2  \otimes\alpha_2 \otimes x_3. && 
	\end{align*}
	Let $j \in  \I_2$ and $g \in \zdos$. Since $\pi(x_j) = 0$, we have by \eqref{eq:twist-mult}  that
	\begin{align*}
		g\cdot_{\sigma^{\pi}} x_j &= \sigma(g, \alpha_1) gx_j, &
		x_j\cdot_{\sigma^{\pi}} g &= \sigma(\alpha_1, g) x_jg. 
	\end{align*}
	Given $i\in \I_2$ we compute
	\begin{align*}
		\alpha_i\cdot_{\sigma^{\pi}} x_1 &= \sigma(\alpha_i, \alpha_1) \alpha_i x_1= 
		p_{i1} q_{i1}  x_1 \alpha_i = p_{i1} q_{i1} p_{1i}^{-1}   x_1 \cdot_{\sigma^{\pi}} \alpha_i .
	\end{align*}
	Hence
	$\alpha_i \rightharpoonup_{\sigma}  x_1 = p_{i1} q_{i1} p_{1i}^{-1}   x_1$.
	Similarly, $\alpha_i \rightharpoonup_{\sigma}  x_3 = p_{i2} q_{i2} p_{2i}^{-1}   x_3$.
	For the action on $x_2$ we set $a_1 =1$, $a_2 = a$. Then 
	\begin{align*}
		\alpha_i\cdot_{\sigma^{\pi}} x_2 &= \sigma(\alpha_i, \alpha_1) \alpha_i x_2= 
		p_{i1}  q_{i1} (x_2 + a_i x_1) \alpha_i  = 
		p_{i1} p_{1i}^{-1} q_{i1} (x_2 + a_i x_1) \cdot_{\sigma^{\pi}} \alpha_i 
	\end{align*}
	and the verification of \eqref{eq:YD-structure-twisted} is complete.
	Therefore the braiding $c^{\sigma}$ of $\Vs_{\sigma} = \toba(\Vs)_{\sigma}(1)$ 
	is determined by $(c^{\sigma}(x_i \otimes x_j))_{i,j\in \I_3} =$
	\begin{align*}
		\begin{pmatrix}
			q_{11} x_1 \otimes x_1&  q_{11}(x_2 + x_1) \otimes x_1& p_{12}p_{21}^{-1} q_{12}  x_3  \otimes x_1
			\\
			q_{11} x_1 \otimes x_2 & q_{11} (x_2 + x_1) \otimes x_2& p_{12}p_{21}^{-1} q_{12}  x_3  \otimes x_2
			\\
			p_{21}p_{12}^{-1} q_{21} x_1 \otimes x_3 & p_{21}p_{12}^{-1} q_{21} (x_2 +a x_1) \otimes x_3&  q_{22} x_3  \otimes x_3
		\end{pmatrix}.
	\end{align*} 
	If we choose $p_{11} = p_{21} = p_{22} = 1$ and $p_{12} = q'_{12} q_{12}^{-1}$, then clearly
	$\Vs_{\sigma} \simeq \Vs'$ as braided vector spaces. Now
	$\toba(\Vs)_{\sigma} \simeq \toba(\Vs_{\sigma})$ by Lemma \ref{lema:twisting-nichols}.
	\epf
	
	If  $q_{11} = q_{22} = 1$, $q = q_{12} = q_{21}^{-1}$ and $\ghost = -2a$, then $\Vs =: \lstr_q( 1, \ghost)$.
	
	\begin{prop}\label{prop:twisting}  Let $q, q' \in \kut$.
		Then $\toba(\lstr_q( 1, \ghost))$ and $\toba(\lstr_{q'}( 1, \ghost))$ are twist-equivalent.
		\qed
	\end{prop}
	
	We now determine the braided Hopf algebra structure of $ \quot{\ghost}{q}$. 
	
	\begin{prop}
		As braided Hopf algebra, $ \quot{\ghost}{q}$ is  twist-equivalent to  the enveloping algebra of a graded nilpotent Lie algebra.
	\end{prop}
	
	\pf  We realize $\toba(\lstr_q( 1, \ghost))$ in $\yd{\Z^{2}}$ by \eqref{eq:YD-structure}; 
	since $\ku x_1$ is a Yetter-Drinfeld submodule of $\toba(\lstr_q( 1, \ghost))$, 
	$ \quot{\ghost}{q}$ is an object in $\yd{\Z^{2}}$. 
	Let $\overline{x}_2$ and $ \overline{z}_{n}$ be the images of $x_2$ and $ z_n$ 
	in $ \quot{\ghost}{q}$, $n \in \I_{0,\ghost}$.
	We claim 
	the vector subspace $\ngo_{q}$ of  $ \quot{\ghost}{q}$ spanned by $\overline{x}_2$ and 
	$ \overline{z}_{n}$, $n \in \I_{0,\ghost}$,
	is  a subobject in $\yd{\Z^{2}}$. Indeed, by \eqref{eq:YD-structure} we have that
	\begin{align*}
		\alpha_{1} \rightharpoonup \overline{x}_2 &=  \overline{x}_2 ,&
		\alpha_{2} \rightharpoonup \overline{x}_2 &= q^{-1} \overline{x}_2,  & \delta (\overline{x}_2) &= \alpha_{1} \otimes \overline{x}_2.
	\end{align*}
	On the other hand, we know by \cite[Lemma 4.2.1]{aah-triang} that
	\begin{align*}
		\alpha_1 \rightharpoonup z_i &= q z_i, & \alpha_2 \rightharpoonup z_i &= q^{-i}z_i , &\delta(z_i) &= \alpha_1^{i} \alpha_2 \otimes z_i,&
		i  &\in \I_{0, \ghost}.
	\end{align*}
	Hence the analogous formulas for 
	$ \overline{z}_{n}$ hold in $ \quot{\ghost}{q}$ and $ \ngo_{q}$ is of diagonal type with braiding given by
	\begin{align*}
		&c(\overline{x}_2 \otimes \overline{x}_2)= \overline{x}_2 \otimes \overline{x}_2, & & c(\overline{x}_2 \otimes \overline{z}_i)= q \overline{z}_i \otimes \overline{x}_2,\\ & c(\overline{z}_i \otimes \overline{x}_2)=q^{-1} \overline{x}_2 \otimes \overline{z}_i, & & c(\overline{z}_i \otimes \overline{z}_j)=q^{i-j} \overline{z}_j \otimes \overline{z}_i,\,\, i,j \in \I_{0, \ghost}.
	\end{align*}
	Let $v_1, v_2$ be primitive elements of  a braided Hopf algebra whose braiding 
	satisfies $c(v_i \otimes v_j) = q_{ij} v_j \otimes v_i$,
	where $q_{ij} \in \kut$ and $q_{12}q_{21} = 1$. 
	A well-known argument shows that $v_1v_2 - q_{12}v_2v_1$ is again primitive.
	Hence  $ \overline{z}_{n}\in \quot{\ghost}{q}$ is primitive, $n \in \I_{0,\ghost}$.   
	When $q=1$, the braiding of $\ngo := \ngo_{1}$ is the usual flip so that $\ngo$ is a  
	nilpotent Lie algebra and $ \quot{\ghost}{1} \simeq U(\ngo)$. 
	
	\medbreak 
	Let $(p_{ij})_{i,j \in \I_2} = \begin{pmatrix}
		1 & q^{-1}\\ 1 & 1\end{pmatrix}$  and $\sigma: \ku\zdos \otimes \ku \zdos \to \ku$ be the invertible 2-cocycle determined by
	$\sigma(\alpha_{i}, \alpha_{j}) =   p_{ij}$ as in Lemma \ref{lem:twisting}. 
	Consider the bosonization $\cK = \quot{\ghost}{q} \# \ku \zdos$; 
	as explained above, $\cK_{\sigma^{\pi}} \simeq  \quot{\ghost}{q}_{\sigma} \# \ku \zdos$.
	Arguing as in the verification of \eqref{eq:YD-structure-twisted} we conclude that
	\begin{align*}
		\alpha_{1} \rightharpoonup_{\sigma} \overline{x}_2 &= \alpha_{2} \rightharpoonup_{\sigma} \overline{x}_2 
		= \overline{x}_2,& 
		\alpha_{1} \rightharpoonup_{\sigma} \overline{z}_0 &= \alpha_{2} \rightharpoonup_{\sigma} \overline{z} =  \overline{z}_{0}.
	\end{align*}
	Thus the action on $\quot{\ghost}{q}_{\sigma}$ is trivial. Now we appeal to \eqref{eq:twist-mult-R}:
	\begin{align*}
		\overline{x}_2 \cdot_{\sigma } \overline{z}_{n} &= q^{-1} \overline{x}_2 \overline{z}_{n}, &
		\overline{z}_{n}\cdot_{\sigma } \overline{x}_2 &= \overline{z}_{n}  \overline{x}_2, &
		\overline{z}_{n}\cdot_{\sigma } \overline{z}_m &= q^{-n}\overline{z}_{n} \overline{z}_m.
	\end{align*}
	We claim that $\overline{x}_2$ and $\widetilde {z}_{n} = q^{-n} \overline{z}_{n}$ in $\quot{\ghost}{q}_{\sigma}$
	satisfy the defining relations of $\quot{\ghost}{1}$. Indeed for $n \in \I_{0,\ghost}$ we have
	\begin{align*}
		\tag{\ref{eq:lstr-rels&11disc-2}}
		&\begin{aligned}
			\widetilde {z}_n \cdot_{\sigma }\widetilde {z}_{n+1}&= q^{-2n-1} \overline{z}_{n} \cdot_{\sigma }\overline{z}_{n+1}
			=  q^{-3n-1} \overline{z}_{n} \overline{z}_{n+1} = 
			q^{-3n-2} \, \overline{z}_{n+1} \overline{z}_{n}
			\\
			&= q^{-2n-1} \, \overline{z}_{n+1} \cdot_{\sigma } \overline{z}_{n}
			= \widetilde{z}_{n+1} \cdot_{\sigma } \widetilde{z}_n;
		\end{aligned}
		\\
		\tag{\ref{eq:-1block+point-z2}}
		& \begin{aligned}
			\overline{x}_2\cdot_{\sigma }\widetilde {z}_n &= q^{-n} \overline{x}_2\cdot_{\sigma }\overline{z}_{n}
			= q^{-n-1} \overline{x}_2 \overline{z}_{n}
			= q^{-n - 1} \left( q\overline{z}_{n}\overline{x}_2 + \overline{z}_{n+1}\right)
			\\ 
			&= q^{-n}\overline{z}_{n}\cdot_{\sigma }\overline{x}_2 + q^{-n - 1}\overline{z}_{n+1}
			= \widetilde{z}_{n}\cdot_{\sigma }\overline{x}_2 +  \widetilde{z}_{n+1};
		\end{aligned}
		\\
		\tag{\ref{eq:lstr-rels&11disc-qserre}}
		& \begin{aligned}
			\overline{x}_2 \cdot_{\sigma } \widetilde {z}_\ghost &= q^{-\ghost} \overline{x}_2\cdot_{\sigma }\overline{z}_{\ghost}
			=q^{-\ghost-1} \overline{x}_2\overline{z}_{\ghost} = q^{-\ghost} \overline{z}_{\ghost}\overline{x}_2 
			= \widetilde {z}_\ghost \cdot_{\sigma} \overline{x}_2.  
		\end{aligned}
	\end{align*}
	It follows now easily that $\quot{\ghost}{q}_{\sigma}$
	is isomorphic to $\quot{\ghost}{1}$ as braided Hopf algebras, i.~e. 
	$ \quot{\ghost}{q}$  and $ \quot{\ghost}{1}$ are  twist-equivalent.
	\epf
	
	Summarizing,  $\toba(\lstr_q( 1, \ghost))$ and $\toba(\lstr_{1}( 1, \ghost))$ are twist-equivalent and 
	there is an extension of braided Hopf algebras 
	\begin{align*}
		0\to \ku [x_1]\to \toba(\lstr_1( 1, \ghost)) \overset{\varpi}{\longrightarrow}  U(\ngo) \to 0
	\end{align*}
	where $\ngo$ is the Lie algebra with basis $\{\mathtt{x},  \mathtt{z}_n:  n \in \I_{0,\ghost}\}$
	and bracket
	\begin{align*}
		[\mathtt{x},  \mathtt{z}_n] &=  \mathtt{z}_{n+1}, \quad n \in \I_{0,\ghost -1},&
		[\mathtt{x},  \mathtt{z}_{\ghost}] &=0, & [\mathtt{z}_n, \mathtt{z}_m] &=0, \quad  n,m \in \I_{0,\ghost}.
	\end{align*}

	\subsection{Isomorphisms}\label{subsec:braid-hopf-isom}
	Let $q, q' \in \kut$ and $\ghost,\ghost' \in \N$. 
	
	\begin{enumerate}[leftmargin=*,label=\rm{(\alph*)}]
		\item Assume that $\toba(\lstr_q( 1, \ghost)) \simeq \toba(\lstr_{q'}( 1, \ghost'))$ as  braided Hopf algebras. Then 
		$\ghost=\ghost'$ and $q=q'$.  Indeed the isomorphism should preserve the space of primitive elements and the braiding by \cite{S}.

		\medbreak
		\item Assume that $\toba(\lstr_q( 1, \ghost)) \simeq \toba(\lstr_{q'}( 1, \ghost'))$ as algebras.
		Then $\ghost=\ghost'$, since $\ghost+3 = \GK \toba(\lstr_q( 1, \ghost))= \GK \toba(\lstr_{q'}( 1, \ghost'))= \ghost'+3$.
		Furthermore if $1 < \ord q = N < \infty$, then $\ord q' = N < \infty$ 
		since  $\toba(\lstr_q( 1, \ghost)) $ has a simple module of dimension $N$.

		\medbreak
		\item However we do not know whether $\toba(\lstr_q( 1, \ghost)) \simeq \toba(\lstr_{q'}( 1, \ghost))$ as algebras
		implies $q=q'$. In particular, it is natural to guess that $\toba(\lstr_{q}( 1, \ghost))$ is isomorphic to $\toba(\lstr_{1}( 1, \ghost))$ as algebras only when $q=1$. We show that this is indeed the case by an argument based on the determination of the 
		finite-dimensional simple modules.
		
	\end{enumerate} 
	
	\medbreak
	Let $R$ be a ring. For brevity we say ideal for two-sided ideal. 
	The set of isomorphism classes of simple $R$-modules is  denoted $\IRR R$. 
	For each $p \in \IRR R$ we fix a representative $N_p$.
	By definition, see \cite{GL}, the closed sets of the Zariski topology on  $\IRR R$
	are the sets 
	\begin{align*}
		\cV(I)  &= \{p\in \IRR R: I\cdot N_p = 0\}, & &\text{$I$ ideal of $R$}.
	\end{align*}

	When $R$ is commutative, $\IRR R = \Irr R$ with this topology
	is naturally equivalent to the maximal spectrum of $R$ with the classical Zariski topology.
	In general $\Irr R$ is a topological space with the induced topology.
	
	\medbreak
	Let $\varphi: R \to S$ be a ring homomorphism and let  $\varphi^t: \IRR S\to \IRR R$ denote the natural map given by 
	induction along $\varphi$. 
	
	\begin{lemma}\label{lema:zariski}
		If $\varphi$ is surjective, then $\varphi^t$ is a closed continuous map.
	\end{lemma}
	
	\pf It suffices to show that for any ideals $I$ of $R$ and $J$ of $S$ we have that
	\begin{align*}
		(\varphi^t)^{-1}\big(\cV(I)\big)  &= \cV(\varphi(I)), & \varphi^t(\cV (J))& = \cV \big(\varphi^{-1}(J)\big).
	\end{align*}
	Here  $\varphi(I)$ is an ideal because $\varphi$ is surjective.
	Since $I\cdot \varphi^t(N_p) = \varphi(I)\cdot N_p$, we have 
	\begin{align*}
		(\varphi^t)^{-1}\big(\cV(I)\big)  =  \{p\in \IRR S: I\cdot \varphi^t(N_p) = 0\} =\cV(\varphi(I)).
	\end{align*}
	Given $p\in \IRR S$ we have $J\cdot N_p = \varphi^{-1}(J)  \cdot\varphi^t(N_p)$ as $\varphi$ is surjective; 
	thus $\varphi^t(\cV (J)) \subset \cV \big(\varphi^{-1}(J)\big)$.
	Also if $q\in \cV \big(\varphi^{-1}(J)\big)$, then $\ker \varphi \cdot N_q = 0$ i.~e. $N_q \in \text{Im\, }\varphi^t$ and 
	the other contention holds. 
	\epf
	
	\begin{prop}
		If $\toba(\lstr_{q}( 1, \ghost)) \simeq \toba(\lstr_{1}( 1, \ghost))$ as algebras, then $q=1$.
	\end{prop}
	
	\pf If $q = 1$, then $\Irr \toba(\lstr_{1}( 1, \ghost))$ is homeomorphic  to the plane with the Zariski topology, by Theorem \ref{thm:irredu-lest} and Lemma \ref{lema:zariski}. Let $q\neq 1$; we may assume that $q$ is not a root of 1. 
	By Theorem \ref{thm:irredu-lest} and Lemma \ref{lema:zariski}, $\Irr \toba(\lstr_{q}( 1, \ghost))$ is homeomorphic  to $\Irr \ku_q[X,Y] = U_1 \cup U_2$  where $U_1$ and $U_2$ are homeomorphic  to $\ku \times 0$ and $0  \times  \ku$ respectively; just apply 
	Lemma \ref{lema:zariski} to the projections $\ku_q[X,Y] \to \ku [X]$ and  $\ku_q[X,Y] \to \ku [Y]$ and Proposition \ref{prop:simple_quantum-plane}. Thus 
	$\Irr \toba(\lstr_{1}( 1, \ghost))$ is not homeomorphic to $\Irr \toba(\lstr_{q}( 1, \ghost))$.
	\epf
	
	\section{Point modules over \texorpdfstring{$\toba(\lstr_q( 1, \ghost))$}{}}\label{section:point-modules}

	\subsection{Point modules}
	
	Let $A=\oplus_{n\in \N_0}A^n$, $A^0 \simeq\ku$, be a graded $\ku$-algebra with $\dim_{\ku} A^n$ finite, $n\in \N$, generated in degree $1$. A {\it point module over $A$} is a (left) graded module $V=\oplus_{n\in \N_0}V^n$ over $A$ such that $V$ is cyclic, generated in degree $0$, and has Hilbert series $h_V(t)=1/(1-t)$, in other words $\dim _{\ku}V_n=1$, $n\in  \N_0 $.
	Point modules, introduced in \cite{ATV}, allow the introduce projective geometry in graded ring theory. See the survey \cite{Ro}.
	If $A$ is strongly noetherian, then the point modules for $A$ are parametrized 
	by a  projective scheme \cite[Corollary E4.12]{AZ},  \cite[Theorem 3.10]{Ro}. 
	Our goal in this Section is to compute the projective scheme parametrizing the point modules over $\toba(\lstr_q( 1, \ghost))$ 
	which we have shown in Section \ref{section:preliminaries} that is strongly noetherian.
	We do this by essentially elementary calculations.
	
	\medbreak
	We first recall the parametrization of point modules over a free associative algebra given in  \cite[Proposition 3.5 ]{Ro}. 
	As usual $(a_0 : a_1: \cdots : a_n)$ with $a_i \in \ku$ denotes  a point of the projective space $\mathbb{P}^n = \mathbb{P}^n(\ku)$.

	\begin{theorem}\label{point-freealgebra}
		Let $A=\ku\langle x_i: i \in \I_{0,n}\rangle$ be the free associative algebra. 
		The isomorphism classes of point modules over $A$  are in bijective correspondence with $\N_0$-indexed sequences of points in $\mathbb{P}^n$, in other words, points in the infinite product $ \prod_{i=0}^{\infty} \mathbb{P}^n$. The correspondence is given by: 
		\begin{align*}
			V=\oplus_{i\in \N_0}\langle v_i\rangle  &\mapsto (P_0,P_1,\ldots)\in \prod_{i=0}^{\infty} \mathbb{P}^n, &
			P_i &:= (a_{0,i}: \cdots :a_{n,i}),
		\end{align*}
		where $x_jv_i  = a_{j,i} \,v_{i+1}$.
	\end{theorem}
	
	Given an homogeneous element  $F$ of the polynomial ring $\ku[X_0, X_1, X_2]$, $\cV(F)$ denotes the 
	projective subvariety of $\bP^2$ of zeros of $F$.
	
	\begin{theorem} \label{thm:point-modules-les}
		The isomorphism classes of point modules over $\toba(\lstr_q( 1, \ghost))$ are parametrized by $\cV(X_0X_2)$. 
	\end{theorem}
	
	The parametrization is given by $V\mapsto P_0$ in the notation of Theorem \ref{point-freealgebra}.
	To prove Theorem \ref{thm:point-modules-les}, observe that  $\cV(X_0X_2) = B \cup C  \cup \{ (0:0:1)\}$ where
	\begin{align*}
		B&:=\{(1:b:0):\,b\in \ku \},&  C&:=\{(0:1:c):\, c\in \ku \}.
	\end{align*}
	
	We deal with the point modules parametrized by $B$ and $C$ in Lemmas \ref{lema:point-part1} and  \ref{lema:point-part2} while
	we show in Lemma \ref{lema:point-part3}  that the  rest corresponds to $(0:0:1)$.  
	
	\smallbreak
	Recall from Lemma \ref{lem:epi-algebra} and Subsection \ref{subsec:quantum-plane} the algebra surjections
	\begin{align*}
		\xymatrix{\toba(\lstr_q( 1, \ghost))  \ar@/_0.7pc/_{\nu_{\ghost}}[drrr] \ar@{->}^(.4){\pi_{\ghost}}[r]
			& \toba(\lstr_q( 1, \ghost-1)) \ar@{->}_{\nu_{\ghost - 1}}[drr]  \cdots 
			& \toba(\lstr_q( 1, 2)) \ar@{->}^{\pi_{2}}[r] \ar@{->}_(.3){\nu_{2}}[dr]& \toba(\lstr_q( 1, 1))\ar@{->}^{\nu_{1}}[d] 
			\\
			&&& \ku_q[X,Y].} 
	\end{align*}
	The associated maps $\pi_{\ghost}^t$, $\pi_{\ghost - 1}^t$, \dots between the varieties of point modules are all isomorphisms, 
	while $\nu_{\ghost}^t$, $\nu_{\ghost - 1}^t$, \dots identify the variety corresponding to the quantum plane, which is $\mathbb P^1$
	by \cite[Example 3.2]{Ro}, with $C  \cup \{ (0:0:1)\}$.

	\subsection{Proof of Theorem \ref{thm:point-modules-les}}
	In the rest of the section $V= \oplus_{i\in \N_0} V_i$ denotes a point module over $\toba(\lstr_q( 1, \ghost))$ 
	with $V_i=\langle v_i\rangle$, $i\in \N_0$. Since  $x_1, x_2, z_0$ have degree one  and $V$ is cyclic,
	there exists $P_i=(a_i: b_i: c_i) \in \mathbb{P}^2$ such that 
	\begin{align}\label{scalars}
		x_1 v_i &= a_iv_{i+1},&  x_2 v_i &= b_iv_{i+1},&  z_0v_i &= c_iv_{i+1}, & i\in \N_0.
	\end{align}
	By Theorem \ref{point-freealgebra}, $V$ is completely determined by $P:=(P_0,P_1,\ldots)\in \prod_{i=0}^{\infty} \mathbb{P}^2$.  
	We start by the following identity in $\toba(\lstr_q( 1, \ghost))$: 
	\begin{equation} \label{lema-aux-formulas1}
		\sum\limits_{i\in \I_{0,\ghost+1}} \textstyle\binom{\ghost+1}{i}(-q)^{i} \,x_2^{\ghost+1-i}z_0x_2^i=0.
	\end{equation}
	
	\pf 
	One proves recursively on $n \in \I_{ \ghost+1}$ that 
	$z_n=\sum\limits_{i \in \I_{0, n}} \textstyle\binom{n}{i}(-q)^{i}\,x_2^{n-i}z_0x_2^i$. The claim follows because of the defining relation \eqref{eq:lstr-rels&11disc-qserre}. \epf
	
	\begin{remark}\label{rem:ai=0} The following are equivalent: 
		\begin{enumerate*}[label=\rm{(\roman*)}]
			\item $a_0 = 0$, \item $a_{i} = 0 $ for some  $i \in \N_0$,  \item $a_{i} = 0 $ for all $i \in \N_0$.
		\end{enumerate*}
	\end{remark}
	
	\pf
	The relations  \eqref{eq:rels B(V(1,2))} and \eqref{eq:lstr-rels&11disc-1} imply that 
	\begin{align}\label{eq:rels-pointmod}
		a_ib_{i+1} -a_{i+1} b_i+\frac{1}{2}a_ia_{i+1} &= 0,& a_{i+1}c_i-qa_ic_{i+1}&=0,& i &\in\N_0.
	\end{align} 
	Assume that $a_i \neq 0$. We claim that $a_{i+1} \neq 0$. 
	Indeed, if $a_{i+1} = 0$, then  \eqref{eq:rels-pointmod} implies that  
	$b_{i+1} = c_{i+1} = 0$, that is $V$ is not cyclic, a contradiction. Similarly if $a_{i+1} \neq 0$ and $a_{i} = 0 $, 
	then $b_{i} = c_{i} = 0$, again a contradiction.
	Hence $a_i=0$ if and only if $a_{i+1}=0$ and the Remark follows.
	\epf

	\begin{lemma} \label{lema:point-part1}
		If $a_0\neq 0$, then $P_i=(1:b_0-i/2:0)$ for all  $i \in \N_0$.
	\end{lemma}
	\pf
	Given $i \in \N_0$, by Remark \ref{rem:ai=0} $a_i \neq 0$,
	hence we can assume that $a_i = 1$. By  \eqref{eq:rels-pointmod},
	$b_{i+1} = b_{i} - \frac{1}{2}$ and $c_{i+1} = q^{-1}c_{i}$. 
	Therefore
	\begin{align*}
		P_i  &= (1: b_0 - i/2 : q^{-i} c_0),&  i &\in \N_0.
	\end{align*}
	It remains to prove that $c_0=0$.  Evaluating both sides of  \eqref{lema-aux-formulas1} on $v_{0}$ and reordering, we have that
	\begin{align} \label{eq-for}
		\sum\limits_{i\in \I_{0,\ghost+1}} \textstyle\binom{\ghost+1}{i} (-q)^{i} \,b_0\cdots b_{i-1}b_{i+1}\cdots b_{\ghost+1}c_i=0.
	\end{align}
	Suppose first that $b_j=0$ for some $j \in  \I_{0, \ghost+1}$, that is $b_0 = j/2$. Then $b_i \neq 0$ for all $i \neq j$. 
	By \eqref{eq-for}, $b_0\cdots b_{j-1}b_{j+1}\cdots b_{\ghost+1}c_j=0$; thus $c_j=0$ and $c_0=0$.  
	
	\medbreak
	Hence we may assume that $b_i\neq 0$, $i\in \I_{0, \ghost+1}$. Set $b:=b_0b_1\cdots b_{\ghost+1} \neq  0$ and $\hat{b}_i:=b/b_i$.
	By \eqref{eq-for} we have that 
	\begin{align}\label{eq-for1}
		\begin{aligned}
			0 &=\sum\limits_{i \in \I_{0,\ghost+1}}\textstyle\binom{ \ghost+1}{i} (-q)^{i}\hat{b}_ic_i 
			=\sum\limits_{i \in \I_{0,\ghost+1}}(-1)^{i}\binom{ \ghost+1}{i}\hat{b}_ic_0\\
			&=bc_0\,\,\sum\limits_{i \in \I_{\ghost+1}}(-1)^{i} \textstyle\binom{\ghost+1}{i} \frac{1}{b_0-i/2}\\
			&=2bc_0\sum\limits_{i \in \I_{0,\ghost+1}}(-1)^{i}\textstyle\binom{\ghost+1}{ i} \frac{1}{2b_0-i}.
		\end{aligned}
	\end{align}

	It is easy to prove by induction on $n$ that
	\begin{align}\label{id-aux}
		\sum\limits_{i \in \I_{0,n}}(-1)^{i} \textstyle\binom{n}{i} \frac{1}{t-i} &= \frac{(-1)^n\, n!}{t(t-1)\ldots(t-n)},&
		n\in \N, \ t &\in \Bbbk \backslash \N_0.
	\end{align}
	
	Applying \eqref{id-aux} to $t=2b_0$ we obtain from \eqref{eq-for1} that
	\[\frac{2bc_0(-1)^{\ghost+1}\,(\ghost+1)!}{2b_0(2b_0-1)\cdots (2b_0-(\ghost+1))}=0.\]
	Hence  $c_0=0$,  consequently  $c_i=0$ for $i\in \N_0$ and $P_i=(1:b_0-i/2:0)$.\epf

	\begin{lemma} \label{lema:point-part2}
		Assume that $a_0=0$ and $b_j\neq 0$ for all  $j\in \N_0$. Then
		\begin{align*}
			P_j &= \Big(0:1:q^{-j}\frac{c_0}{b_0}\Big),&  j &\in \N_0.
		\end{align*}
	\end{lemma}
	\begin{proof} By Remark \ref{rem:ai=0}, $a_j = 0$, hence $P_j = (0:1:\frac{c_j}{b_j})$,  $j \in \N_0$. 
		Set 
		\begin{align}\label{eq:def-lambda-beta}
			\ld{j}{0} &:=\frac{c_j}{b_j}, &  \ld{j}{n + 1} &:=\ld{j}{n}-q\ld{j+1}{n},& \beta_{j,n}&:=b_jb_{j+1}\cdots b_{j+n},
		\end{align}
		for $j,n\in \N_0$.
		Applying repeatedly \eqref{eq:-1block+point-z2}, we have
		\begin{align}\notag
			z_1v_j  = (b_{j+1}c_j - qc_{j+1} b_j)v_{j+2} &=\beta_{j,1}(\textstyle\frac{c_j}{b_j}-q\textstyle\frac{c_{j+1}}{b_{j+1}})v_{j+2}=\beta_{j,1}\ld{j}{1}v_{j+2} ,
			\\ \label{action-zn}
			z_nv_j &= \beta_{j,n}\ld{j}{n}v_{j+n+1}, \quad j\in \N_0,\,\,\,n\in \N.
		\end{align}
		From \eqref{eq:lstr-rels&11disc-qserre} and \eqref{action-zn} follows that \[\beta_{j,\ghost+1}\left(\ld{j}{\ghost}-q\ld{j+1}{\ghost}\right)=0.\] 
		By \eqref{eq:lstr-rels&11disc-2}, $z_{n+1}z_n-qz_nz_{n+1}=0$. Thus \eqref{action-zn}  implies also that \[\beta_{j,2n+1}\left(\ld{j}{n}\ld{j+n+1}{n+1}-q\ld{j}{n+1}\ld{j+n+2}{n}\right)=0,\quad  n\in \I_{0, \ghost-1}.\]
		Since all $\beta_{j,i}$'s are $\neq 0$ we  are led  to deal with  the following  systems of polynomial equations
		on the variables $\ldv{j}{0}$, $j \in \N_0$.  Define recursively
		\begin{align}\label{eq:def-lambda-beta-var}
			\ldv{j}{n + 1} =\ldv{j}{n}-q\ldv{j+1}{n}.
		\end{align}
		
		We consider for each  $\ghostvar\in \N$ the infinite system 
		\begin{align}\label{nonl-sys}
			\tag*{($\mathscr{S}_\ghostvar$)}
			&\begin{aligned}
				\left\{ \begin{array}{ll}
					\ldv{j}{\ghostvar}-q\ldv{j+1}{\ghostvar}=0, & \\[.8em]
					\ldv{j}{n}\ldv{j+n+1}{n+1}-q\ldv{j}{n+1}\ldv{j+n+2}{n}=0, &
				\end{array} \right. 
			\end{aligned}&
			j &\in \N_0, & n &\in \I_{0, \ghostvar-1}.
		\end{align}

		\begin{claim*}  
			The  system \ref{nonl-sys}  has a unique solution $\big(\ldsol{j}{0}\big)_{j\in \N_0}$ for each  $x \in \ku$, namely
			\begin{align}\label{eq:system-solutions}
				\ldsol{j}{0} &= q^{-j}x, &  j&\in \N_0. 
			\end{align}
		\end{claim*}
		
		It is easy to see that  \eqref{eq:system-solutions} is a solution of 
		$(\mathscr{S}_{\ghostvar})$. For the converse we  proceed by induction on $\ghostvar$. 
		Let $\big(\ldsol{j}{0}\big)_{j\in \N_0}$ be a solution of  $(\mathscr{S}_{1})$.
		Then 
		\begin{align}\label{nonl-sys-ghost1}
			&\begin{aligned}
				\left\{ \begin{array}{ll}
					\ldsol{j}{0}-2q\ldsol{j+1}{0}+q^2\ldsol{j+2}{0}=0, & \\[.8em]
					\ldsol{j}{0}\ldsol{j+1}{0} -  2q\ldsol{j}{0}\ldsol{j+2}{0} + q^2\ldsol{j+1}{0}\ldsol{j+2}{0}=0.&
				\end{array} \right.
			\end{aligned} & j &\in \N_0,
		\end{align}
		by \eqref{eq:def-lambda-beta-var}. 
		The second equation of \eqref{nonl-sys-ghost1} minus the first multiplied by $\ldsol{j+1}{0}$
		gives
		$(\ldsol{j+1}{0})^2 - \ldsol{j}{0}\ldsol{j+2}{0}  = 0$;
		replacing $\ldsol{j+2}{0}$ by  $\frac{-1}{q^{2}} \left(\ldsol{j}{0}-2q\ldsol{j+1}{0}\right)$ we get
		\begin{align*}
			(\ldsol{j+1}{0})^2 + q^{-2} (\ldsol{j}{0})^2 - 2q^{-1} \ldsol{j}{0}\ldsol{j+1}{0} &=  \left(\ldsol{j+1}{0} - q^{-1} \ldsol{j}{0}\right) ^2= 0.
		\end{align*}
		That is, $\ldsol{j+1}{0} = q^{-1} \ldsol{j}{0}$ for all $j \in \N_0$; this implies \eqref{eq:system-solutions}. 
		
		\medbreak
		Assume now  that the claim holds for $\ghostvar > 0$. 
		Let $\big(\ldsol{j}{0}\big)_{j\in \N_0}$ be a solution of  $(\mathscr{S}_{\ghostvar + 1})$.
		By \eqref{eq:def-lambda-beta-var}, the first equation gives 
		\begin{align*}
			\ldsol{j+2}{\ghostvar} &= 2q^{-1}\ldsol{j+1}{\ghostvar}-q^{-2}\ldsol{j}{\ghostvar}, & j &\in \N_0.
		\end{align*}
		Then it is easy to prove recursively that 
		\begin{align}\label{for-recurs}
			\ldsol{j + h}{\ghostvar} &= h q^{1-h}\ldsol{j+1}{\ghostvar}-(h -1)q^{-h}\ldsol{j}{\ghostvar},& h &\geq 2.
		\end{align}
		When $n=\ghostvar$, the second equation of $(\mathscr{S}_{\ghostvar + 1})$  together with \eqref{eq:def-lambda-beta-var} says that
		\begin{align*}
			\ldsol{j}{\ghostvar}\ldsol{j+\ghostvar+1}{\ghostvar} - 2q\ldsol{j}{\ghostvar }\ldsol{j + \ghostvar + 2}{\ghostvar}  
			+ q^2 \ldsol{j + 1}{\ghostvar }\ldsol{j + \ghostvar + 2}{\ghostvar} &=0,  & j &\in \N_0.
		\end{align*}
		Plugging  \eqref{for-recurs} into the previous equality we see that 
		\begin{align*}
			(\ghostvar+2)\left(q^{-\ghostvar-1}\left(\ldsol{j}{\ghostvar}\right)^2-2q^{-\ghostvar}\ldsol{j}{\ghostvar}\ldsol{j+1}{\ghostvar}+q^{-\ghostvar+1}\left(\ldsol{j+1}{\ghostvar}\right)^2\right)=0.
		\end{align*}
		That is, $\left(\ldsol{j}{\ghostvar}-q\ldsol{j+1}{\ghostvar}\right)^2=0$. Hence  we have that   $\ldsol{j}{0}$, 
		$j\in \N_0$, is a solution of  $(\mathscr{S}_{\ghostvar})$.
		By the inductive hypothesis,   $\ldsol{j}{0}=q^{-j}\ldsol{0}{0}$ and  the Claim follows.
		
		\medbreak
		Since $\big(\frac{c_j}{b_j}\big)$ is a solution of $(\mathscr{S}_{\ghost})$ by the discussion above,  the Claim implies that
		$\dfrac{c_j}{b_j}=q^{-j}\dfrac{c_0}{b_0}$, $j\in \N_0$. The Lemma follows.
	\end{proof}
	
	We next proceed with the remaining posibility.
	
	\begin{lemma} \label{lema:point-part3}
		Assume that $a_0=0$ and $b_i =  0$ for some $i\in \N_0$. Then 
		\begin{align*}
			P_j &=(0:0:1), & j &\in \N_0.
		\end{align*}
	\end{lemma}
	\begin{proof}
		We set $z_nv_i=\af{i}{n}v_{i+n+1}$, $i\in \N_0$, $n \in \I_{0,\ghost}$. 
		Recall that $a_i = 0$ for all $i\in \N_0$ by Remark \ref{rem:ai=0}; thus $b_i$ and $\zeta^{(0)}_i=c_i$ 
		could not be both 0, as $V$ is cyclic.  
		The proof of \eqref{eq:well-known} is easy and follows a well-known pattern:
		\begin{align}\label{eq:well-known}
			z_n \overset{\eqref{eq:-1block+point-z2}}{=} x_2 z_{n-1} - q z_{n-1} x_2 
			=\sum_{k \in \I_{0, n}} \textstyle\binom{n}{k} (-q)^k x_2^{n-k} z_0 x_2^k, \quad n\in \I_{\ghost}.
		\end{align}
		Evaluating these identities  at $v_i$, $i \in \N_0$, we get for $ n\in \I_{\ghost}$:
		\begin{align} \label{eq-sis-0}
			\af{i}{n} & =\af{i}{n-1}b_{i+n}-qb_i\af{i+1}{n-1} = \sum_{k \in \I_{0, n}} \textstyle\binom{n}{k} (-q)^k  \af{i+k}{0}b^{(k)}_{i,n}
			\\
			\notag \text{where } \quad & b^{(k)}_{i,n} = b_i b_{i+1} \cdots b_{i+k-1} b_{i+k+1}\cdots b_{i+n} = \prod_{h\in \I_{0, n}, h\neq k} b_{i+h}.
		\end{align}
		
		Evaluating \eqref{eq:lstr-rels&11disc-2}, respectively \eqref{eq:lstr-rels&11disc-qserre}, at $v_i$
		and plugging in appropriate instances of \eqref{eq-sis-0},  we get for  $i\in \N_0$ and $n\in \I_{0,\ghost-1}$
		\begin{align} \label{eq-sis-1}
			\af{i}{n}\af{i+n+1}{n}b_{i+2n+2}&-2q\af{i}{n}\af{i+n+2}{n}b_{i+n+1} + q^2 \af{i+1}{n}\af{i+n+2}{n} b_i=0,
			\\ 
			\label{eq-sis-2} &\af{i}{\ghost}b_{i+\ghost+1}-qb_i\af{i+1}{\ghost}=0.
		\end{align}

		\bigbreak\emph{We fix for the remaining of the proof $i\in \N_0$ such that $b_i=0$.}

		\begin{paso}\label{claim1} We have  $b_{i+1}=0$ if and only if $b_{i+2}=0$. If, in addition, $b_{i+1}=0$, then $b_j=0$, $j\in \N_0$.
		\end{paso}
		
		\noindent Since $c_i = \af{i}{0}\neq 0$ it follows from \eqref{eq-sis-1} that $\af{i+1}{0}b_{i+2} - 2q\af{i+2}{0}b_{i+1} = 0$. 
		Thus $b_{i+1}=0$ if and only if $b_{i+2}=0$. Consequently, if $b_{i+1}=0$, then $b_{i+\ell}=0$, $\ell \geq 2$.
		
		Assume that there exists $i\in \N_0$ such that $b_i=b_{i+1}=0$.
		Let $t\in \N_0$ be the smallest one such that $b_{t}= b_{t+1}= 0$. If $t>0$ then \eqref{eq-sis-1} implies that
		$\af{t-1}{0}\af{t}{0}b_{t+1} - 2q\af{t-1}{0}\af{t+1}{0}b_{t} + q^2\af{t}{0}\af{t+1}{0}b_{t-1} = 0$. 
		Since $b_t=b_{t+1}=0$, we get that $b_{t-1}\af{t}{0}\af{t+1}{0}=0$, a contradiction because $\af{t}{0}\neq 0\neq\af{t+1}{0}$. 
		Hence $t=0$ and the second part of the claim follows from the first.

		\begin{paso} \label{claim3}
			Either $b_j=0$  for all  $j\in \N_0$ or else $b_{i+m} \neq 0$
			for all $m\in \N$.
		\end{paso}

		Assume that the first possibility does not hold. We shall prove by induction that $b_{i+2n+1} \neq 0$ and $b_{i+2n+2} \neq 0$ for all 
		$n\in \N_0$. When $n = 0$,  $b_{i+1} \neq 0$ and  $b_{i+2} \neq 0$ by Step \ref{claim1}.
		Let $n\in \N$ and suppose that $b_{i+1} \ldots b_{i+2n}\neq 0$. 
		We claim  that $b_{j+2n+1} \neq 0$ and $b_{j+2n+2} \neq 0$.
		Since $b_i = 0$, $\af{i}{0} \neq 0$  we have 
		\begin{align*}
			\af{i}{n}  & \overset{\eqref{eq-sis-0}}{=} \sum_{k \in \I_{0, n}} \textstyle\binom{n}{k} (-q)^k  \af{i+k}{0}b^{(k)}_{i,n}
			= \af{i}{0}b^{(0)}_{i,n} =  \af{i}{0} b_{i+1}\ldots b_{i+n}\neq 0.
		\end{align*} 
		
		Then \eqref{eq-sis-1}implies that 
		\begin{align}\label{eq-sis-1-part}
			\af{i+n+1}{n}b_{i+2n+2}&-2q \af{i+n+2}{n}b_{i+n+1} =0.
		\end{align}
		If  $b_{i+2n+2} = 0$, then $\af{i+2n+2}{0}\neq 0$ and
		\begin{align*}
			0& \overset{\eqref{eq-sis-1-part}}{=}\af{i+n+2}{n} = \sum_{k \in \I_{0, n}} \textstyle\binom{n}{k} (-q)^k  \af{i+n+2+k}{0} \, b^{(k)}_{i+n+2,n}
			\\
			& = (-q)^n  \af{i+2n+2}{0} \, b_{i+n+2}b_{i+n+3} \dots b_{i+2n+1} \implies b_{i+2n+1}  = 0.
		\end{align*}
		By Step \ref{claim1} $b_j = 0$ for all $j\in \N_0$, contradicting the assumption.
		
		Similarly assume that   $b_{i+2n+1} = 0$. Then $\af{i+2n+1}{0}\neq 0$ and
		\begin{align*}
			\af{i+n+1}{n} &= (-q)^n  \af{i+2n+1}{0} \, \, b_{i+n+1} \dots b_{i+2n} 
			\\
			\af{i+n+2}{n} &= n (-q)^{n-1}  \af{i+2n+1}{0}  \, b_{i+n+2} \dots b_{i+2n} b_{i+2n+2},
			\\
			\implies 0 &\overset{\eqref{eq-sis-1-part}}{=} (1 - 2n) (-q)^n  \af{i+2n+1}{0} \, \, b_{i+n+1} \dots b_{i+2n} b_{i+2n+2} 
			\\
			\implies 0 &= b_{i+2n+2}.
		\end{align*}
		Again $b_j = 0$ for all $j\in \N_0$ by Step \ref{claim1}, a contradiction.
		The Step is proved.
		
		\bigbreak
		To finish the proof of  the Lemma, we just observe that 
		\begin{align*}
			0 \overset{\eqref{eq-sis-2}}{=} &\af{i}{\ghost}b_{i+\ghost+1}  = \sum_{k \in \I_{0, \ghost}} \textstyle\binom{\ghost}{k} (-q)^k  \af{i+k}{0}b^{(k)}_{i,\ghost}b_{i+\ghost+1} 
			\\ = &\af{i}{0} b_{i+1}b_{i+2}b_{i+3}\dots   b_{i,\ghost}b_{i+\ghost+1} 
		\end{align*}
		Hence $b_{i+3}\ldots b_{i+\ghost+1}=0$. 
		Step \ref{claim3} implies that $b_{j}=0$ for all $j\in \N_0$ and the Lemma follows.
	\end{proof}

	\section*{Declaration}
	
	\subsection*{Data Availability Statement}
	Data sharing is not applicable to this article as no datasets were generated or analysed during the current study.

\end{document}